\documentclass[oneside,english]{amsart}
\usepackage[T1]{fontenc}
\usepackage[latin9]{inputenc}
\usepackage{geometry}
\usepackage{amsthm}
\usepackage{amssymb}

\makeatletter
\numberwithin{equation}{section}
\numberwithin{figure}{section}
\theoremstyle{plain}
\newtheorem{thm}{\protect\theoremname}
  \theoremstyle{plain}
  \newtheorem{lem}[thm]{\protect\lemmaname}
  \theoremstyle{remark}
  \newtheorem*{rem*}{\protect\remarkname}
  \theoremstyle{plain}
  \newtheorem{cor}[thm]{\protect\corollaryname}
  \theoremstyle{plain}
  \newtheorem{conjecture}[thm]{\protect\conjecturename}
  \theoremstyle{plain}
  \newtheorem{prop}[thm]{\protect\propositionname}

\makeatother

\usepackage{babel}
  \providecommand{\conjecturename}{Conjecture}
  \providecommand{\corollaryname}{Corollary}
  \providecommand{\lemmaname}{Lemma}
  \providecommand{\propositionname}{Proposition}
  \providecommand{\remarkname}{Remark}
\providecommand{\theoremname}{Theorem}

\begin{document}

\title{A remark on a conjecture of Buzzard-Gee and the Cohomology of Shimura
Varieties }

\author{Christian Johansson}
\address{Department of Mathematics, Imperial College London, London SW7 2AZ, UK}
\email{h.johansson09@imperial.ac.uk}

\begin{abstract}

We compare the conjecture of Buzzard-Gee on the association of Galois representations to $C$-algebraic automorphic representations with the conjectural description of the cohomology of Shimura varieties due to Kottwitz, and the reciprocity law at infinity due to Arthur. This is done by extending Langlands's representation of the $L$-group associated with a Shimura datum to a representation of the $C$-group of Buzzard-Gee. The approach offers an explanation of the explicit Tate twist appearing in Kottwitz's description.

\end{abstract}

\maketitle

\section*{Introduction}

A central theme in modern algebraic number theory is the connection
between automorphic forms and Galois representations. This link has
its origins in class field theory, and evolved from the work of Weil,
Taniyama, Shimura, Eichler, Serre and Deligne amongst others before
its scope was vastly expanded by Langlands in the late 1960's and
70's, linking the emerging theory of automorphic representations to
Grothendieck's conjectural theory of motives, where Galois representations
appear as realizations. For automorphic representations of ${\rm GL}_{n}$,
the first precise conjectures were given by Clozel in \cite{Clo}
where he defined a notion of algebraicity and conjectured that there
should be motives (hence Galois representations) associated with algebraic
automorphic representations.

In \cite{BG1}, Buzzard and Gee generalize the notion of algebraicity
and (a weak form of) Clozel's conjectures to an arbitrary connected
reductive group $G$ over a number field $F$. Perhaps somewhat surprisingly,
they define two notions of algebraicity. On the one hand they define
a notion of $L$-algebraicity, and given an $L$-algebraic automorphic
representation $\pi$ of $G$ they conjecture the existence of an
$\ell$-adic Galois representation 
\[
\rho_{\pi}^{L}\,:\, Gal(\overline{F}/F)\rightarrow{^{L}G(\overline{\mathbb{Q}}_{\ell}})
\]
They also define a notion of $C$-algebraicity which generalizes Clozel's
definition of algebraicity. Buzzard and Gee (following an idea of
Deligne) define a canonical $\mathbb{G}_{m}$-extension $\widetilde{G}$
of $G$ and given a $C$-algebraic automorphic representation $\pi$
of $G$, they canonically construct an $L$-algebraic automorphic
representation $\widetilde{\pi}$ of $\widetilde{G}$ and thus conjecturally
get an associated Galois representation 
\[
\rho_{\pi}=\rho_{\widetilde{\pi}}^{L}\,:\, Gal(\overline{F}/F)\rightarrow{^{L}\widetilde{G}(\overline{\mathbb{Q}}_{\ell})}={^{C}G(\overline{\mathbb{Q}}_{\ell}})
\]
Since the beginnings of the theory, an important test case for any
conjecture on the association of Galois representation to automorphic
representations has been the case of Shimura varieties, starting with
the modular curve. Langlands initiated a program to compute the ($\ell$-adic)
cohomology of arbitrary Shimura varieties (with coefficients in certain
local systems) and later on Kottwitz gave a very precise conjectural
description of these cohomology groups (\cite{Kot1}). The automorphic
representations that contribute to this description are called cohomological
and they are the one of most important classes of $C$-algebraic automorphic
representations.

In this short note we wish to test the Buzzard-Gee conjecture on the
association of Galois representations with $C$-algebraic automorphic
forms in the case of cohomological automorphic representations on
groups admitting a Shimura variety. Namely, given a cohomological
automorphic representation $\pi$, we show that Kottwitz's description
implies (or rather is equivalent to) that the composition of $\rho_{\pi}$
with a certain fixed algebraic representation $r_{C}$ of $^{C}G$
occurs in the cohomology (for the precise statement, see Conjecture
\ref{Conj, modified}). The representation $r_{C}$ is an extension
to $^{C}G$ of a representation of $^{L}G$ (originally defined by
Langlands in \cite{Lan1}) that is used in Kottwitz's description.
Our minor reformulation has the slight advantage that it explains
(or perhaps internalizes) a somewhat mysterious Tate twist by half
the dimension of the Shimura variety that occurs in Kottwitz's description. 

The structure of this note is as follows: In \S\ref{sec: 1} we define
the representation $r_{C}$ and in \S\ref{sec: 2} we prove Corollary
\ref{cor: main} which is the main result needed to relate the Buzzard-Gee
conjecture and Kottwitz's description of the cohomology of Shimura
varieties. This uses only some basic Lie theory and the definition
of a Shimura datum. Finally in \S\ref{sec:Comparison-with-Kottwitz}
we recall Kottwitz's description together with a related result of
Arthur (\cite{Art1}) and make the comparison.

Finally, the author would like to thank his PhD supervisors Kevin Buzzard and Toby Gee for asking him the question 
treated in this paper and for valuable comments on the first draft, Jack Shotton for comments on a later draft and Alexander Paulin for a useful conversation. He also wishes to thank the Engineering and Physical Sciences Research Council for supporting him throughout his doctoral studies. Moreover, it is a pleasure to
thank the Fields Institute, where part of the work for this paper was done, for their support and hospitality as well
as for the excellent working conditions provided.

\section{\label{sec: 1}The representation $r_{C}$ }

In this section we will set up some notation, recall the notion of
a $C$-group from \cite{BG1} and define a representation $r_{C}$
which will be an extension of a representation of Langlands that we
will denote $r_{L}$ (whose definition we will recall below). For simplicity,
all our dual groups and $L$-groups will be complex algebraic groups
and complex pro-algebraic groups, respectively, and will be identified
with their $\mathbb{C}$-points. In \S\ref{sec:Comparison-with-Kottwitz}
we will fix an isomorphism $\iota\,:\,\mathbb{C}\cong\overline{\mathbb{Q}}_{\ell}$,
where $\ell$ is a fixed rational prime. Given a connective reductive
group $G$ over a field $F$ with $L$-group $^{L}G=Gal(\overline{F}/F)\ltimes\widehat{G}$
(we will always use the Galois form of the $L$-group),
we may use $\iota$ to consider the associated pro-algebraic group
$^{L}G(\overline{\mathbb{Q}}_{\ell})$ over $\overline{\mathbb{Q}}_{\ell}$,
identified with its $\overline{\mathbb{Q}}_{\ell}$-points, and by
abuse of notation we will denote by $\iota$ the map $^{L}G\rightarrow{^{L}G(\overline{\mathbb{Q}}_{\ell}})$
induced by $\iota$. 

Let $G$ be a connected reductive group over $\mathbb{Q}$ and let
$X$ be a $G(\mathbb{R})$-conjugacy class of homomorphisms 
\[
h\,:\,\mathbb{S}\rightarrow G_{\mathbb{R}}
\]
of real algebraic groups, where $\mathbb{S}$ denotes Deligne's torus
${\rm Res}_{\mathbb{R}}^{\mathbb{C}}{\rm GL}_{1/\mathbb{C}}$. The
pair $(G,X)$ is required to satisfy Deligne's axioms for a Shimura
datum (see \cite{Del1} \S 2.1.1). We will let $E$ denote the reflex field of $(G,X)$
and we denote by $d$ the complex dimension of $X$. Given $h\in X$, we let $\mu_{h}$ denote the corresponding cocharacter of $G_{\mathbb{C}}$
defined by 
\[
\mu_{h}(z)=h_{\mathbb{C}}(z,1)
\]
where we have fixed an isomorphism $\mathbb{S}(\mathbb{C})\cong\mathbb{C}^{\times}\times\mathbb{C}^{\times}$
(and $\mathbb{S}(\mathbb{R})=\mathbb{C}^{\times}$ embeds by $z\mapsto(z,\bar{z})$).
We will follow the conventions of \cite{Del1} for associating real
Hodge structures with representations of $\mathbb{S}$ (i.e. the $(p,q)$-space
is the subspace where $\mathbb{S}(\mathbb{C})$ acts by $(z,w)\mapsto z^{-p}w^{-q}$)
and the normalization of the reciprocity map of local class field
theory (uniformizers go to geometric Frobenii). Fix a pinning $(T,B)$
of $G$ with corresponding pinning $(\widehat{T},\widehat{B})$ of
$\widehat{G}$ that is fixed by the Galois action (we suppress the choice of elements in the root spaces for the simple roots). This gives us the
notion of dominant weights and coweights and positive roots and coroots
for $G$ and $\widehat{G}$. We let $\mu$ denote an element
of the conjugacy class of $(-\mu_{h})_{h\in X}$ which is antidominant.
The representation $r_{L}$ is defined in two steps. $\mu$ defines
an antidominant integral weight $\widehat{\mu}$ of $\widehat{G}$
and hence a unique irreducible representation 
\[
r\,:\,\widehat{G}\rightarrow{\rm Aut}(V_{\mu})
\]
that has $\widehat{\mu}$ as an extreme weight. We extend this to our desired representation 
\[
r_{L}\,:\,^{L}G_{E}=\Gamma_{E}\ltimes\widehat{G}\rightarrow{\rm Aut}(V_{\mu})
\]
by letting $\Gamma_{E}=Gal(\overline{\mathbb{Q}}/E)$ act trivially
on the weight space of $\widehat{\mu}$. 

Buzzard and Gee have defined the notion of a $C$-group $^{C}G$ for
$G$ (\cite{BG1} Definition 5.3.2), which is defined to be the $L$-group
$^{L}\widetilde{G}$ of a canonical extension 
\[
1\rightarrow{\rm GL_{1}\rightarrow\widetilde{G}\rightarrow G\rightarrow1}
\]
of $G$. Proposition 5.3.3 of \cite{BG1} gives
a central isogeny 
\[
\widehat{G}\times\mathbb{C}^{\times}\rightarrow\widehat{\widetilde{G}}
\]
which is Galois-equivariant and whose kernel is generated by the element
$(e,-1)$ of order $2$. Here $e=\widehat{\chi}(-1)$, where $\chi$
is the sum of positive roots of $G$ and $\widehat{\chi}$
is the sum of the positive coroots of $\widehat{G}$. We wish to extend
$r_{L}$ to a representation $r_{C}$ of $^{C}G_{E}=\Gamma_{E}\ltimes\widehat{\widetilde{G}}$.
Let $\left\langle -,-\right\rangle $ denote the pairing between the
character and cocharacter lattices of $G$ and $\widehat{G}$ to $\mathbb{Z}$.
Part 1) of the following Lemma is well known.
\begin{lem}
\label{lem:easy}1) $\left\langle \chi,\mu\right\rangle =-d$.

2) $e$ acts as $(-1)^{-d}$ on $V_{\mu}$.\end{lem}
\begin{proof}
1) By axiom (2.1.1.1) for a Shimura datum (in the notation of \cite{Del1} \S 2.1.1) and the choice of $\mu$, $\mu(z)$ acts as $z^{-1}$ or $1$ on the positive
root spaces (i.e. $\left\langle \alpha,\mu\right\rangle =-1$ or $0$ for
each positive root $\alpha$) and as $z$ or $1$ on the negative root spaces. The result follows since there are $d$
roots on which $\mu(z)$ acts as $z^{-1}$, by axiom (2.1.1.2) for a Shimura datum (\cite{Del1} \S 2.1.1). 

2) Since $e$ is central it acts on $V_{\mu}$ by a scalar and hence
it is enough to compute the action on the weight space of $\widehat{\mu}$. By part 1), the action is by 
\[
\widehat{\mu}(e)=\widehat{\mu}(\widehat{\chi}(-1))=(-1)^{\left\langle \widehat{\mu},\widehat{\chi}\right\rangle }=(-1)^{\left\langle \chi,\mu\right\rangle }=(-1)^{-d}.
\]

\end{proof}
Let us now extend $r$ to a representation of $\widehat{G}\times{\rm GL}_{1}$
by 
\[
(g,z)\mapsto r(g)z^{-d}.
\]
By Lemma \ref{lem:easy} $(e,-1)$ acts trivially so this defines
a representation of $\widehat{\widetilde{G}}$. We may extend this
to our desired representation $r_{C}$ of $^{C}G_{E}$ on $V_{\mu}$
by letting $\Gamma_{E}$ act trivially on the weight space of $\widehat{\mu}$;
by the $\Gamma_{E}$-equivariance of the isogeny $\widehat{G}\times{\rm GL}_{1}\rightarrow\widehat{\widetilde{G}}$
this defines an extension of $r_{L}$.

\section{\label{sec: 2}The result}

Let $\tau$ be an irreducible algebraic representation of $G$, and
let $\pi$ be a discrete automorphic representation of $G$ which
is cohomological for $\tau$, i.e. there is an integer $i$ such that
$H^{i}(\mathfrak{g}_{\infty},K_{\infty},\pi_{\infty}\otimes\tau)\neq0$,
where $\mathfrak{g}_{\infty}$ is the Lie algebra of $G_{\mathbb{R}}$,
$K_{\infty}\subseteq G(\mathbb{R})$ is the stabilizer of some fixed
$x\in X$ and $H^{i}(\mathfrak{g}_{\infty},K_{\infty},-)$ is the
relative Lie algebra cohomology. Lemma 7.2.2 of \cite{BG1} says that
$\pi$ is $C$-algebraic (\cite{BG1}, Definition 3.1.2).

\begin{rem*}
$\pi$ is also $C$-arithmetic (\cite{BG1}, Definition 3.1.4). This
follows from the argument in \S2.3 of \cite{BR1} taking coefficients
in $\mathcal{F}_{\tau}$ (defined near the end of the first paragraph
of \S\ref{sec:Comparison-with-Kottwitz} of this paper) instead of the trivial
local system.
\end{rem*}

Let $\pi$ be a $C$-algebraic automorphic representation on $G$.
We will briefly review the construction of an $L$-algebraic (\cite{BG1},
Definition 3.1.1) automorphic representation $\widetilde{\pi}$ of
$\widetilde{G}$, canonically associated to $\pi$ (see \cite{BG1},
discussion before Conjecture 5.3.4). Given $\pi$, one uses the canonical
map $\widetilde{G}\rightarrow G$ to pull back $\pi$ to a $C$-algebraic
automorphic representation $\pi^{\prime}$ (\cite{BG1} Lemma 5.1.2).
The central isogeny $\widehat{G}\times{\rm GL}_{1}\rightarrow\widehat{\widetilde{G}}$
mentioned in the previous section is dual to an isogeny $(c,\xi)\,:\,\widetilde{G}\rightarrow G\times{\rm GL}_{1}$
defined over $\mathbb{Q}$. Then 
\[
g\mapsto|\xi(g)|^{1/2}
\]
is a character $\widetilde{G}(\mathbb{A})\rightarrow\mathbb{C}^{\times}$,
and we define $\widetilde{\pi}$ to be the twist $\pi^{\prime}\otimes|\xi(-)|^{1/2}$.
It is an $L$-algebraic automorphic representation of $\widetilde{G}$. 

For a local or global field $F$ of characteristic $0$ we let $\mathcal{L}_{F}$
denote the Langlands group of $F$. It carries a canonical surjection
$\mathcal{L}_{F}\twoheadrightarrow W_{F}$, where $W_{F}$ is the
Weil group of $F$. When $F$ is local, we have $\mathcal{L}_{F}=W_{F}$
if $F$ is archimedean (with the canonical map being the identity)
and we take $\mathcal{L}_{F}=W_{F}\times{\rm SL}_{2}(\mathbb{C})$
if $F$ is non-archimedean (the canonical map being projection onto
the first factor). When $F$ is global, this group only exists conjecturally.
We will use it only for motivation (in particular to make the comparison
with \cite{Kot1}); in the end all conjectures and results may be
stated using only $\mathcal{L}_{F}$ for local fields. 

Let $|-|$ denote the composition of $\mathcal{L}_{F}\twoheadrightarrow W_{F}$
with the norm character $W_{F}\rightarrow\mathbb{C}^{\times}$.

\begin{lem}
\label{lem: key lemma}1) Let $p$ be a finite prime where $\pi$
and $G$ are unramified and let $\phi_{p}\,:\,\mathcal{L}_{\mathbb{Q}_{p}}\rightarrow{^{L}G}$
be the $L$-parameter (Satake parameter) associated with $\pi_{p}$.
Then the $L$-parameter $\widetilde{\phi}_{p}\,:\,\mathcal{L}_{\mathbb{Q}_{p}}\rightarrow{^{C}G}$
associated to $\widetilde{\pi}_{p}$ is given by 
\[
\widetilde{\phi}_{p}(w)=(\phi_{p}(w),|w|^{1/2})
\]
 where we abuse notation and denote by $(\phi_{p}(w),|w|^{1/2})$
the image of $(\phi_{p}(w),|w|^{1/2})\in{^{L}G}\times\mathbb{C}^{\times}$
in $^{C}G$.

2) Let $\phi_{\mathbb{C}}\,:\,\mathcal{L}_{\mathbb{C}}\rightarrow{^{L}G}$
be the restriction of the $L$-parameter associated with $\pi_{\infty}$.
Then the restriction of the $L$-parameter of $\widetilde{\pi}_{\infty}$
to $\mathcal{L}_{\mathbb{C}}$ is given by 
\[
\widetilde{\phi}_{\mathbb{C}}(w)=(\phi_{\mathbb{C}}(w),|w|^{1/2}).
\]
\end{lem}

\begin{proof}
1) This follows from the construction of the unramified Local Langlands
correspondence as described in \S 10.4 of \cite{Bor} and is implicit
in \cite{BG1} (in the derivation of Conjecture 5.3.4 from Conjecture
3.2.1); we will content ourselves with giving a brief sketch of the
proof. Let $\widetilde{T}$ resp. $\widetilde{B}$ be the inverse images
of $T$ resp. $B$ under $\widetilde{G}\rightarrow G$. $\pi_{p}$
occurs as a subquotient of some unramified principal series ${\rm Ind}_{B}^{G}\chi$
(normalized induction), where $\chi\,:\, T(\mathbb{Q}_{p})\rightarrow\mathbb{C}^{\times}$
is an unramified character. Since parabolic induction behaves well
with respect to pullback $\widetilde{G}(\mathbb{Q}_{p})\rightarrow G(\mathbb{Q}_{p})$,
$\pi^{\prime}$ occurs as a subquotient of ${\rm Ind}_{\widetilde{B}}^{\widetilde{G}}\chi^{\prime}$,
where $\chi^{\prime}$ is the composition of $\chi$ with $\widetilde{T}\rightarrow T$.
This implies that the $L$-parameter of $\pi^{\prime}$ is $w\mapsto(\phi_{p}(w),1)$
since the correspondence for unramified characters of unramified tori
is functorial. To get the $L$-parameter of $\widetilde{\pi}$ we
use that the unramified Local Langlands correspondence behaves well
with respect to unramified twists, cf. e.g. Remark 2.2.1 of \cite{BG1}
(though it is perhaps simpler to deduce this from the construction
in \cite{Bor} since parabolic induction behaves well under twists,
rather than using the Satake isomorphism directly as is done in \cite{BG1}).

2) is proved in exactly the same way as 1), though the details are
simpler, using the construction of the complex Local Langlands correspondence
(see \cite{Bor} \S 11.4). Here one deduces the behavior with respect
to twists from same property for parabolic induction as indicated
in the end of the proof of 1).
\end{proof}

The following simple consequence is our main technical result.

\begin{cor}
\label{cor: main}With notation as in Lemma \ref{lem: key lemma},
we have 
\[
(r_{L}\circ\phi_{v})\otimes|-|^{-d/2}=r_{C}\circ\widetilde{\phi}_{v}
\]
for $v$ finite where $\pi$ and $G$ are unramified or $v=\mathbb{C}$.\end{cor}
\begin{proof}
By Lemma \ref{lem: key lemma} and the definition of $r_{C}$ we have
\[
r_{C}(\widetilde{\phi}_{v}(w))=r_{C}(\phi_{v}(w),|w|^{1/2})=r_{L}(\phi_{v}(w))|w|^{-d/2}
\]
as desired. 
\end{proof}

\section{Comparison with Kottwitz's Conjecture and Arthur's Result\label{sec:Comparison-with-Kottwitz}}

Let us try to briefly describe the conjecture of Kottwitz and the
result of Arthur. For simplicity, we will follow Arthur's presentation
in \S9 of \cite{Art1}. Recall our Shimura datum $(G,X)$. If $K\subseteq G(\mathbb{A}^{\infty})$
is a compact open subgroup we will let $Sh_{K}$ denote the canonical
model over $E$ of the corresponding Shimura variety and $\overline{Sh}_{K}$
its minimal compactification (\cite{Pin1}). We write $Sh$ for the inverse
system $(Sh_{K})_{K}$ and $\overline{Sh}$ for $(\overline{Sh}_{K})_{K}$.
The reflex field $E$ comes with an embedding into $\mathbb{C}$ and
we let $\overline{\mathbb{Q}}$ denote the algebraic closure of $\mathbb{Q}$
inside $\mathbb{C}$. Let $\ell$ be a fixed rational prime; we fix
an isomorphism $\iota\,:\,\overline{\mathbb{Q}}_{\ell}\cong\mathbb{C}$,
thus we may view $\overline{\mathbb{Q}}$ as a subfield of $\overline{\mathbb{Q}}_{\ell}$.
Let $\tau$ be an algebraic representation of $G$. Then we may associate
to $\tau$ a sheaf $\mathcal{F}_{\tau}$ which is either a constructible
local system of $\mathbb{Q}$-vector spaces on $Sh_{K}(\mathbb{C})^{an}$,
a smooth $\overline{\mathbb{Q}}_{\ell}$-sheaf on $Sh_{K}$ or a vector
bundle with flat connection on $Sh_{K}$. They satisfy various compatibilities
with respect to transition maps in $Sh$ and comparison theorems for
the relevant cohomology theories; we will abuse notation and denote
them all by $\mathcal{F}_{\tau}$, as well as the canonical extension
of the smooth $\overline{\mathbb{Q}}_{\ell}$-sheaf $\mathcal{F}_{\tau}$
on $Sh_{K}$ to $\overline{Sh}_{K}$ (defined via the theory of perverse
sheaves). There is a geometric action of $G(\mathbb{A}^{\infty})$
on $Sh$ and $\overline{Sh}$. For any suitable cohomology theory
$H$ we write
\[
H^{\ast}(Sh_{K},\mathcal{F}_{\tau})=\bigoplus_{i}H^{i}(Sh_{K},\mathcal{F}_{\tau})
\]
 
\[
H^{\ast}(Sh,\mathcal{F}_{\tau})=\underset{\rightarrow}{{\rm lim}}\,H^{\ast}(Sh_{K},\mathcal{F}_{\tau})\in Mod(G(\mathbb{A}^{\infty})\times?)
\]
where the $?$ signifies that the cohomology theory may carry extra structure.
We use analogous notation for $\overline{Sh}$.

The starting point for the analysis of the cohomology of Shimura varieties
is the result 
\[
H_{(2)}^{\ast}(Sh,\mathcal{F}_{\tau})=\bigoplus_{\pi}m(\pi)\pi^{\infty}\otimes H^{\ast}(\mathfrak{g}_{\infty},K_{\infty},\pi_{\infty}\otimes\tau)\in Mod(G(\mathbb{A}^{\infty}))
\]
of Borel and Casselman \cite{BoCa}, where $H_{(2)}$ denotes $L^{2}$-cohomology,
$m(\pi)$ is the multiplicity of $\pi$ in the discrete spectrum of
$L^{2}(G(\mathbb{Q})\backslash G(\mathbb{A})^{1})$ and $H^{\ast}(\mathfrak{g}_{\infty},K_{\infty},-)=\bigoplus_{i}H^{i}(\mathfrak{g}_{\infty},K_{\infty},-)$
is total relative Lie algebra cohomology. To $\pi$ such that $m(\pi)\neq0$
one should conjecturally be able to attach an $A$-parameter $\psi$
and we may, assuming Conjecture 8.1 of \cite{Art1}, rewrite the above
as 

\begin{equation}
H_{(2)}^{\ast}(Sh,\mathcal{F}_{\tau})=\bigoplus_{\psi}\bigoplus_{\pi\in\Pi_{\psi}}m_{\psi}(\pi)\pi^{\infty}\otimes H^{\ast}(\mathfrak{g}_{\infty},K_{\infty},\pi_{\infty}\otimes\tau)\label{eq: 2}
\end{equation}

where $\Pi_{\psi}$ is the $A$-packet associated to $\psi$ and $m_{\psi}(\pi)$
is a certain multiplicity. Put 
\[
V_{\psi}=\bigoplus_{\pi_{\infty}\in\Pi_{\psi_{\infty}}}H^{\ast}(\mathfrak{g}_{\infty},K_{\infty},\pi_{\infty}\otimes\tau).
\]
Associated with $\psi$ is a group $S_{\psi}$ defined on p. 52 of
\cite{Art1} (as well as local versions $S_{\psi_{v}}$ for any place $v$ of $\mathbb{Q}$). $V_{\psi}$ carries a representation of $S_{\psi_{\infty}}$, a Hodge structure and a Lefschetz decomposition
(\cite{Art1} pp. 59-61). Arthur then defines, for each $\pi$, a vector
space $U_{\pi}$ that only depends on $\pi^{\infty}$ and carries
an action of $S_{\psi}$, an action of $S_{\psi}$ on $U_{\pi}\otimes V_{\psi}$ and rewrites (\ref{eq: 2}) as 
\begin{equation}
H_{(2)}^{\ast}(Sh,\mathcal{F}_{\tau})=\bigoplus_{\psi}\bigoplus_{\pi^{\infty}\in\Pi_{\psi^{\infty}}}\pi^{\infty}\otimes(U_{\pi}\otimes V_{\psi})_{\epsilon_{\psi}}\label{eq: 3}
\end{equation}
where $\epsilon_{\psi}$ is a certain sign character of $S_{\psi}$
and $(-)_{\epsilon_{\psi}}$ denotes the subspace where $S_{\psi}$
acts as $\epsilon_{\psi}$. This uses Conjecture 8.5 of \cite{Art1}
for the multiplicity $m_{\psi}(\pi)$. By various comparison theorems
together with Zucker's conjecture equation (\ref{eq: 3}) holds for
$\ell$-adic intersection cohomology of the minimal compactification
\[
H_{et,\ell}^{\ast}(\overline{Sh},\mathcal{F}_{\tau})=\bigoplus_{\psi}\bigoplus_{\pi^{\infty}\in\Pi_{\psi^{\infty}}}\pi^{\infty}\otimes(U_{\pi}\otimes V_{\psi})_{\epsilon_{\psi}}
\]
as $G(\mathbb{A}^{\infty})$-modules, after applying $\iota$ (as
$\iota$ is fixed, we will omit it from the notation). $(U_{\pi}\otimes V_{\psi})_{\epsilon_{\psi}}$
carries an action of $S_{\psi}$ (by $\epsilon_{\psi}$), and for
$L^{2}$-cohomology it carries a Hodge structure and a Lefschetz decomposition,
whereas for $\ell$-adic intersection cohomology it carries a representation
of $Gal(\overline{\mathbb{Q}}/E)$ and a Lefschetz decomposition.
We wish to describe this extra structure. 

To that end, let $\psi\,:\,\mathcal{L}_{\mathbb{Q}}\times{\rm SL}_{2}(\mathbb{C})\rightarrow{^{L}G}$
be the $A$-parameter attached to $\pi$. We write $\psi_{E}$ for
the restriction to $\mathcal{L}_{E}\times{\rm SL}_{2}(\mathbb{C})$
and let $\phi_{E}$ denote the associated $L$-parameter $\mathcal{L}_{E}\rightarrow{^{L}G}$,
defined by 
\[
\phi_{E}(w)=\psi_{E}\left(w,\left(\begin{array}{cc}
|w|^{1/2} & 0\\
0 & |w|^{-1/2}
\end{array}\right)\right).
\]
Kottwitz composes $\phi_{E}$ with $r_{L}$ to obtain an $L$-parameter
$\phi_{E,r_{L}}$ and conjectures that there should exist a motive
$M_{\psi}$ over $E$ whose $\ell$-adic \'etale realization $H_{et,\ell}(M_{\psi})$
satisfies 
\[
WD_{\iota}(H_{et,\ell}(M_{\psi})|_{Gal(\overline{E}_{v}/E_{v})})=(\phi_{E,r_{L}}|-|^{-d/2})|_{\mathcal{L}_{E_{v}}}
\]
for each place $v\nmid\ell\infty$ of $E$ (where $WD_{\iota}$ means
take the associated complex Weil-Deligne representation, using $\iota$).
Moreover its de Rham realization $H_{dR,v}(M_{\psi})$ with respect
to a place $v\mid\infty$ of $E$ should satisfy 
\[
H_{dR,v}(M_{\psi})=(\phi_{E,r_{L}}|-|^{-d/2})|_{\mathcal{L}_{\overline{E}_{v}}}.
\]
Here $H_{dR,v}(M_{\psi})$, as a Hodge structure, carries an action
of $\mathbb{C}^{\times}\cong W_{\overline{E}_{v}}$. In both cases
the Lefschetz decomposition on $H_{et}(M_{\psi})$ and $H_{dR}(M_{\psi})$
should be given by $r_{L}\circ\psi_{E}|_{{\rm SL}_{2}(\mathbb{C})}$.
Kottwitz verifies (\cite{Kot1} p. 200) that these actions of $r_{L}\circ\psi_{E}$
(hence of $\phi_{E,r_{L}}$ and its twist) commute with the action
of $S_{\psi}$ on $V_{\psi}$, and hence that we get induced structures
on $(U_{\pi}\otimes V_{\psi})_{\epsilon_{\psi}}$. For the finite places we then have

\begin{conjecture}
\label{Kottwitz}(Kottwitz, \cite{Kot1} p. 201). Assume that the
derived group $G^{der}$ of $G$ is simply connected and that the
maximal $\mathbb{R}$-split torus of the center $Z(G)$ of $G$ is
$\mathbb{Q}$-split. Then the $Gal(\overline{\mathbb{Q}}/E)$-representation
\[
W_{et,\ell}^{\ast}(\pi^{\infty})={\rm Hom}_{G(\mathbb{A}^{\infty})}(\pi^{\infty},H_{et,\ell}^{\ast}(\overline{Sh},\mathcal{F}_{\tau}))
\]
is equal to 
\[
\bigoplus_{\psi\,:\,\pi^{\infty}\in\Pi_{\psi^{\infty}}}(U_{\pi}\otimes H_{et}(M_{\psi}))_{\epsilon_{\psi}}.
\]
\end{conjecture}

\begin{rem*}
1) Kottwitz does much more than stating the conjecture. Assuming Arthur's
conjectures on the discrete spectrum (i.e. Conjectures 8.1 and 8.5 of \cite{Art1}), some conjectures on transfer
and a formula for the number of points modulo primes of good reduction
for $Sh$, Kottwitz computes the contribution of the Euler characteristic $\sum(-1)^{i}H_{c,et,\ell}^{i}(Sh,\mathcal{F}_{\tau})$ of compact support cohomology to the Euler characteristic $\sum(-1)^{i}IH_{et,\ell}^{i}(\overline{Sh},\mathcal{F}_{\tau})$.
Since Kottwitz has to work with $\sum(-1)^{i}IH_{et,\ell}^{i}(\overline{Sh},\mathcal{F}_{\tau})$
(the fundamental technique used being a comparison between geometric
sides of trace formulas), his conjecture looks slightly different
to what we have written above.

2) Of course, Kottwitz's conjecture has been (partially) proven in many cases
of PEL type A or C, and for groups related to inner forms of $GL_{2}$ over a totally
real field.
\end{rem*}

At infinity we have the following:

\begin{thm}
\label{thm: Arthur}(Arthur, \cite{Art1} Proposition 9.1). As representations
of $\mathbb{C}^{\times}$, ${\rm SL}_{2}(\mathbb{C})$ (i.e. the
Hodge structure resp. the Lefschetz decomposition) and $S_{\psi_{\infty}}$ we have
\[
V_{\psi}=H_{dR,v}(M_{\psi})
\]
for $v$ the place associated to the canonical embedding $E\hookrightarrow\mathbb{C}$.
\end{thm}

\begin{rem*} 
1) The way we have described it, it is perhaps not obvious that the above statement makes sense without assuming some conjectures, but  it may be formulated entirely without reference to the global $A$-parameter $\psi$ and the motive $M_{\psi}$, using only the local $A$-parameter $\psi_{\infty}$. The definition of $V_{\psi}$ then only depends on $\Pi_{\psi_{\infty}}$, which exists non-conjecturally (\cite{VZ}), and so does $H_{dR,v}(M_{\psi})$ if we define it by the desiderata outlined above, rewritten in a local form (i.e. that the Hodge structure is $(\phi_{\infty,r_{L}}|-|^{-d/2})|_{\mathcal{L}_{\overline{E}_{v}}}$ where $\phi_{\infty}$ is $L$-parameter associated with $\psi_{\infty}$; the Lefschetz decomposition is $r_{L}\circ\psi_{\infty}|_{{\rm SL}_{2}(\mathbb{C})}$, and the action of $S_{\psi_{\infty}}$ is via $r_{L}$).

2) This identifies $V_{\mu}$ and $V_{\psi}$ as
vector spaces and hence gives a formula for the dimension of $V_{\psi}$.
\end{rem*}

The next conjecture is then a consequence of Conjectures 8.1 and 8.5 of \cite{Art1}, as decribed above: 

\begin{conjecture}
\label{Arthur}(Arthur, \cite{Art1}) As representations
of $\mathbb{C}^{\times}$ and ${\rm SL}_{2}(\mathbb{C})$ (i.e. the
Hodge structure and the Lefschetz decomposition) 
\[
W_{dR}^{\ast}(\pi^{\infty})={\rm Hom}_{G(\mathbb{A}^{\infty})}(\pi^{\infty},H_{(2)}^{\ast}(Sh,\mathcal{F}_{\tau}))
\]
is equal to 
\[
\bigoplus_{\psi\,:\,\pi^{\infty}\in\Pi_{\psi^{\infty}}}(U_{\pi}\otimes H_{dR,v}(M_{\psi}))_{\epsilon_{\psi}}
\]
for $v$ the place associated to the canonical embedding $E\hookrightarrow\mathbb{C}$. \end{conjecture}

We now wish to recast this story using the ideas of Buzzard and Gee and the representation $r_{C}$. Since our $\pi$ are $C$-algebraic, Buzzard and
Gee conjecture (\cite{BG1} Conjecture 5.3.4) that there exists a
Galois representation 
\[
\rho_{\pi,\ell}=\rho_{\pi,\ell,\iota}\,:\, Gal(\overline{\mathbb{Q}}/\mathbb{Q})\rightarrow{^{C}G(\overline{\mathbb{Q}}_{\ell}})
\]
satisfying a list of desiderata, the most important for us being that,
for finite primes $p\neq\ell$ such that $\pi$ is unramified, $\rho_{\pi,\ell}|_{W_{\mathbb{Q}_{p}}}$
is $\widehat{\widetilde{G}}(\overline{\mathbb{Q}}_{\ell})$-conjugate
to 
\[
w\mapsto\iota((\phi_{\pi_{p}}(w),|w|^{1/2})).
\]
In other words, $\rho_{\pi,\ell}$ is the Galois representation associated
with the $L$-algebraic automorphic representation $\widetilde{\pi}$
according to Conjecture 3.2.1 of \cite{BG1} (since $w\mapsto(\phi_{\pi_{p}}(w),|w|^{1/2})$
is the Satake parameter of $\widetilde{\pi}_{p}$). We remark that
$\rho_{\pi,\ell}$ depends only on $\pi^{\infty}$ and the $L$-packet
of $\pi_{\infty}$ (but see the remark below). Let $\rho_{E,\pi,\ell}$
denote the restriction of $\rho_{\pi,\ell}$ to $Gal(\overline{\mathbb{Q}}/E)$.
With notation as above, Corollary \ref{cor: main} gives us
\begin{prop}
\label{prop: 6}
Assume that $M_{\psi}$ exists as above. Then

1) $H_{et,\ell}(M_{\psi})=r_{C}\circ\rho_{E,\pi,\ell}$ as representations of $Gal(\overline{\mathbb{Q}}/E)$.

2) $H_{dR,v}(M_{\psi})=r_{C}\circ\widetilde{\phi}_{\mathbb{C}}$ as representations of $\mathbb{C}^{\times}$.
\end{prop}

\smallskip

\begin{rem*}
1) Assume Langlands functoriality (say in the weak form of Conjecture 6.1.1 of \cite{BG1}). Since $\widetilde{\pi}$ is $L$-algebraic the transfer ${\Pi_{E}}$ of $\widetilde{\pi}$ to $\rm{GL}_{N/E}$ (using $r_{C}$ and base change $E/\mathbb{Q}$ ; here $N=\rm{dim}\,V_{\mu}$) is $L$-algebraic (\cite{BG1} Lemma 6.1.2) with $L$-parameter $r_{C}\circ \widetilde{\phi}_{E}$ and can be taken to be isobaric. Thus one sees from Proposition \ref{prop: 6} that $M_{\psi}$ is the motive conjecturally associated with $\Pi_{E}\otimes |\cdot |^{(1-N)/2} $ by Clozel (\cite{Clo} Conjecture 4.5). 

2) We have been somewhat imprecise in the above since $\rho_{\pi,\ell}$
may not be uniquely determined by $\pi$ (see \cite{BG1} Remark 3.2.4
for a discussion and some examples). One could (or should) instead
speak of all possible $\rho_{\pi,\ell}$ satisfying the desiderata;
the statement that $\rho_{\pi,\ell}$ only depends on $\pi^{\infty}$
and the $L$-packet of $\pi_{\infty}$ should be interpreted in this
way. Of course, after composing with a representation of $^{C}G$,
the composition only depends on $\pi$ at its unramified finite places. 
\end{rem*}

We can now state the minor variations of Conjectures \ref{Kottwitz}
and \ref{Arthur} and Theorem \ref{thm: Arthur} given by Proposition \ref{prop: 6}:

\begin{conjecture}
\label{Conj, modified}The $Gal(\overline{\mathbb{Q}}/E)$-representation
\[
W_{et,\ell}^{\ast}(\pi^{\infty})={\rm Hom}_{G(\mathbb{A}^{\infty})}(\pi^{\infty},H_{et,\ell}^{\ast}(\overline{Sh},\mathcal{F}_{\tau}))
\]
is equal to 
\[
\bigoplus_{\psi\,:\,\pi^{\infty}\in\Pi_{\psi^{\infty}}}(U_{\pi}\otimes(r_{C}\circ\rho_{E,\pi,\ell}))_{\epsilon_{\psi}}.
\]
\end{conjecture}

\medskip

\begin{thm}
 We have

\[
V_{\psi}=r_{C}\circ\widetilde{\phi}_{\mathbb{C}}
\]
as representations of $\mathbb{C}^{\times}$.
\end{thm}

\medskip

\begin{conjecture}
As representations of $\mathbb{C}^{\times}$ (i.e. Hodge structures)
\[
W_{dR}^{\ast}(\pi^{\infty})={\rm Hom}_{G(\mathbb{A}^{\infty})}(\pi^{\infty},H_{(2)}^{\ast}(Sh,\mathcal{F}_{\tau}))
\]
is equal to 
\[
\bigoplus_{\psi\,:\,\pi^{\infty}\in\Pi_{\psi^{\infty}}}(U_{\pi}\otimes(r_{C}\circ\widetilde{\phi}_{\mathbb{C}}))_{\epsilon_{\psi}}.
\]
\end{conjecture}

Thus we see that the conjectures of Buzzard and Gee are consistent
in a natural way with the known and conjectural properties of the
cohomology of Shimura varieties, and offers an explanation for the
twist by $-d/2$.
\begin{rem*}
We should remark that another possible way of explaining this Tate
twist is via the theory of weights for perverse sheaves. Namely, it
is natural to consider $\mathcal{F}_{\tau}[d]$ (shift defined by $-[d]^{i}=-^{i+d}$)
because it is a (pure) perverse sheaf. The shift $[d]$ lowers the
weight of the sheaf and its cohomology by $d$, which has the same
effect as a Tate twist by $d/2$ (undoing the Tate twist by $-d/2$
above) in terms of weights.\end{rem*}


\begin{thebibliography}{Kot1}
\bibitem[Art]{Art1} Arthur, J. : Unipotent automorphic representations
I. Conjectures. Ast\'erisque 171\textendash{}172 (1989) 13\textendash{}71.

\bibitem[BR]{BR1} Blasius, D., Rogawski, J. : Zeta-functions of Shimura
varieties. Motives (Seattle, WA, 1991), Proc. Sympos. Pure Math. 55,
AMS, Providence, RI (1994) 525-571.

\bibitem[Bor]{Bor} Borel, A. Automorphic $L$-functions. Automorphic
Forms. Representations and L-functions. Proc. Sympos. Pure Math. 33,
Part II, AMS, Providence, RI (1979) 27-61.

\bibitem[BC]{BoCa} Borel, A., Casselman, W. : $L^{2}$-cohomology
of locally symmetric manifolds of fi{}nite volume. Duke Math. J. 50
(1983), 625-647.

\bibitem[BG]{BG1} Buzzard K. M., Gee, T. : The conjectural connections
between automorphic representations and Galois representations. To
appear in Proceedings of the LMS Durham Symposium 2011. Available
at http://www2.imperial.ac.uk/\textasciitilde{}tsg/

\bibitem[Clo]{Clo} Clozel, L. Motifs et formes automorphes: applications
du principe de fonctorialit\'e. Automorphic forms, Shimura varieties,
and L-functions, Proceedings of the Ann Arbour conference, Perspect.
Math. 10, Academic Press, Boston, MA (1990) 77-159.

\bibitem[Del]{Del1} Deligne, P. : Vari\'et\'es de Shimura : interpr\'etation
modulaire, et techniques de construction de mod\`eles canoniques. Automorphic
forms, representations and L-functions, Proc. Sympos. Pure Math. 33,
Part II, AMS, Providence, RI (1979) 247-290.

\bibitem[Kot]{Kot1} Kottwitz, R. : Shimura varieties and $\lambda$-adic
representations. Automorphic forms, Shimura varieties and L-functions,
Proceedings of the Ann Arbour conference, Perspect. Math. 10, Academic
Press, Boston, MA (1990) 161-209.

\bibitem[Lan]{Lan1} Langlands, R. P. : Automorphic representations,
motives and Shimura varieties. Ein M\"archen. Automorphic Forms. Representations
and L-functions. Proc. Sympos. Pure Math. 33, Part II, AMS, Providence,
RI (1979) 205-246.

\bibitem[Mil]{Mil1} Milne, J. S. : Introduction to Shimura varieties.
Harmonic Analysis, the Trace Formula and Shimura Varieties, AMS, Providence,
RI (2005) 265-378. Also available at http://www.jmilne.org/math/articles/

\bibitem[Pin]{Pin1} Pink, R. : Arithmetical compactifi{}cation of
mixed Shimura varieties, Ph.D. Thesis, Bonner Mathematische Schriften
209 (1989)

\bibitem[VZ]{VZ} Vogan, D., Zuckerman, G. : Unitary representations with non-zero cohomology, Compositio Math. no. 53 (1984) 51-90

\end{thebibliography}
\end{document}